\documentclass[11pt, twoside, leqno]{article}

\usepackage{amssymb}
\usepackage{amsmath}
\usepackage{amsthm}
\usepackage{color}
\usepackage{mathrsfs}

\allowdisplaybreaks

\newcommand{\bR}{{\mathbb R}}

\newcommand{\R}{{\mathbb R}}

\def\sign{{\mbox{\small\rm  sign}}\,}

\def\rr{{\mathbb R}}
\def\rn{{{\rr}^n}}
\def\zz{{\mathbb Z}}

\def\nn{{\mathbb N}}

\def\cx{{\mathcal X}}

\def\mi{{\mathrm I}}

\def\fz{\infty}

\def\dz{\delta}

\def\lf{\left}
\def\r{\right}

\def\ls{\lesssim}
\def\gs{\gtrsim}

\def\noz{\nonumber}

\def\st{\subset}
\def\com{\complement}
\def\bh{\backslash}

\def\supp{\mathop\mathrm{\,supp\,}}

\def\esup{\mathop\mathrm{\,ess\,sup\,}}

\def\BMO{{\mathop\mathrm{BMO}(\rn)}}
\def\bmow{{\mathop\mathcal{BMO}_w(\rn)}}
\def\mv{{\mathcal M_\phi}}

\newtheorem{theorem}{Theorem}[section]
\newtheorem{lemma}[theorem]{Lemma}
\newtheorem{proposition}[theorem]{Proposition}

\theoremstyle{definition}
\newtheorem{definition}[theorem]{Definition}
\newtheorem{remark}[theorem]{Remark}

\numberwithin{equation}{section}

\begin{document}

\arraycolsep=1pt

\title{\bf\Large Weighted Endpoint Estimates for Commutators of
Calder\'on-Zygmund Operators
\footnotetext{\hspace{-0.35cm} 2010 {\it
Mathematics Subject Classification}.
Primary 47B47; Secondary 42B20, 42B30, 42B35.
\endgraf {\it Key words and phrases}. Calder\'on-Zygmund operator, commutator, Muckenhoupt weight,
$\mathop\mathrm{BMO}$ space, Hardy space.
\endgraf
The second author is supported by Vietnam National Foundation for Science and Technology Development (Grant No. 101.02-2014.31).
The third author is supported by the National
Natural Science Foundation  of China
(Grant Nos. 11571039 and 11361020),
the Specialized Research Fund for the Doctoral Program
of Higher Education of China (Grant No. 20120003110003) and
the Fundamental Research
Funds for Central Universities of China (Grant Nos. 2014KJJCA10).}}
\author{Yiyu Liang, Luong Dang Ky and Dachun Yang\,\footnote{Corresponding author}}
\date{ }

\maketitle

\vspace{-0.8cm}

\begin{center}
\begin{minipage}{13cm}
{\small {\bf Abstract}\quad
Let $\delta\in(0,1]$ and $T$ be a $\delta$-Calder\'on-Zygmund operator.
Let $w$ be in the Muckenhoupt class $A_{1+\delta/n}({\mathbb R}^n)$ satisfying
$\int_{{\mathbb R}^n}\frac {w(x)}{1+|x|^n}\,dx<\infty$.
When $b\in{\rm BMO}(\mathbb R^n)$,
it is well known that the commutator $[b, T]$ is not bounded from $H^1(\mathbb R^n)$
to $L^1(\mathbb R^n)$ if $b$ is not a constant function.
In this article, the authors find out a proper subspace
${\mathop\mathcal{BMO}_w({\mathbb R}^n)}$
of $\mathop\mathrm{BMO}(\mathbb R^n)$ such that,
if $b\in {\mathop\mathcal{BMO}_w({\mathbb R}^n)}$, then $[b,T]$ is bounded from the
weighted Hardy space $H_w^1(\mathbb R^n)$ to the weighted Lebesgue
space $L_w^1(\mathbb R^n)$.
Conversely, if $b\in{\rm BMO}({\mathbb R}^n)$ and the commutators of the
classical Riesz transforms $\{[b,R_j]\}_{j=1}^n$
are bounded from $H^1_w({\mathbb R}^n)$ into $L^1_w(\R^n)$,
then $b\in {\mathop\mathcal{BMO}_w({\mathbb R}^n)}$.
}
\end{minipage}
\end{center}

\section{Introduction}\label{s1}

\hskip\parindent
Given a function $b$ locally integrable on $\mathbb R^n$ and a classical
Calder\'on-Zygmund operator $T$, we consider the linear commutator $[b, T]$
defined by setting, for smooth, compactly supported functions $f$,
$$[b, T](f)=bT(f) - T(bf).$$

A classical result of Coifman et al. \cite{CRW} states
that the commutator $[b,T]$ is bounded on $L^p(\mathbb R^n)$ for $p\in(1,\infty)$,
when $b\in{\rm BMO}(\mathbb R^n)$.
Moreover, their proof does not rely on a weak type $(1, 1)$ estimate for $[b, T]$.
Indeed, this operator is more singular than the
associated Calder\'on-Zygmund operator since it fails, in general,
to be of weak type $(1, 1)$, when $b$ is in ${\rm BMO}(\mathbb R^n)$.
Moreover, Harboure et al. \cite[Theorem (3.1)]{HST} showed that $[b,T]$ is bounded from
$H^1(\mathbb R^n)$ to $L^1(\mathbb R^n)$ if and only if
$b$ equals to a constant almost everywhere.
Although the commutator $[b,T]$ does not map continuously,
in general,  $H^1(\mathbb R^n)$ into $L^1(\mathbb R^n)$,
following P\'erez \cite{Pe}, one can find a subspace $\mathcal H^1_b(\mathbb R^n)$
of $H^1(\mathbb R^n)$ such that $[b,T]$ maps continuously
$\mathcal H^1_b(\mathbb R^n)$ into $L^1(\mathbb R^n)$.
Very recently, Ky \cite{Ky} found  the {\sl largest subspace of $H^1(\mathbb R^n)$}
such that all commutators $[b,T]$ of Calder\'on-Zygmund operators are bounded
from this subspace into $L^1(\mathbb R^n)$.
More precisely, it was showed in \cite{Ky} that there exists a bilinear operators
$\mathfrak R:= \mathfrak R_T$ mapping continuously
$H^1(\mathbb R^n)\times{\rm BMO}(\mathbb R^n)$ into $L^1(\mathbb R^n)$ such that,
for all  $(f,b)\in H^1(\mathbb R^n)\times{\rm BMO}(\mathbb R^n)$, we have
\begin{equation}\label{abstract 1}
[b,T](f)= \mathfrak R(f,b) + T(\mathfrak S(f,b)),
\end{equation}
where $\mathfrak S$ is a bounded bilinear operator from
$H^1(\mathbb R^n)\times{\rm BMO}(\mathbb R^n)$ into $L^1(\mathbb R^n)$
which is independent of $T$.
The bilinear decomposition (\ref{abstract 1}) allows ones to give a
general overview of all known endpoint estimates;
see \cite{Ky} for the details.

For the weighted case, when $b\in{\rm BMO}(\R^n)$,
\'Alvarez et al. \cite{abkp} proved that the commutator $[b,T]$
is bounded on the weighted Lebesgue space $L_w^p(\rn)$
with $p\in (1,\fz)$ and $w\in A_p(\rn)$,
where $A_p(\rn)$ denotes the class of Muckenhoupt weights.
Similar to the unweighted case, $[b,T]$ may not be bounded
from the weighted Hardy space
$H_w^1(\rn)$ into the weighted Lebesgue space $L_w^1(\rn)$
if $b$ is not a constant function.
Thus, a natural question is whether there exists a non-trivial subspace
of $\BMO$ such that, when $b$ belongs to this subspace,
the commutator $[b,T]$ is bounded from $H_w^1(\rn)$ to $L_w^1(\rn)$.

The purpose of the present paper is to give an answer for the above question.
To this end, we first recall the definition of the Muckenhoupt weights.
A non-negative measurable function $w$ is said to belong to
the {\it class of Muckenhoupt weight
$A_q(\rn)$} for $q\in[1,\fz)$,
denoted by $w\in A_q(\rn)$ if, when $q\in (1,\fz)$,
\begin{equation}\label{w-1}
[w]_{A_q(\rn)}:=\sup_{B\subset\rn}\frac{1}{|B|}\int_Bw(x)\,dx
\lf\{\frac{1}{|B|}\int_B
[w(y)]^{-q'/q}\,dy\r\}^{q/q'}<\fz,
\end{equation}
where $1/q+1/q'=1$, or, when $q=1$,
\begin{equation}\label{w-2}
[w]_{A_1(\rn)}:=\sup_{B\subset\rn}\frac{1}{|B|}\int_B w(x)\,dx
\lf(\esup_{y\in B}[w(y)]^{-1}\r)<\fz.
\end{equation}
Here the suprema are taken over all balls $B\subset\rn$.
Let 
$$A_\fz(\rn):=\bigcup_{q\in[1,\fz)}A_q(\rn).$$

Let $w\in A_\infty(\R^n)$ and $q\in(0,\infty]$. If $q\in(0,\infty)$,
then we let $L^q_w(\R^n)$ be the space of all measurable functions $f$ such that
\begin{equation}
\|f\|_{L^q_w(\rn)}:= \lf\{\int_{\R^n} |f(x)|^q w(x)\,dx\r\}^{1/q}<\infty.
\end{equation}
When $q=\infty$, $L^\infty_w(\R^n)$ is defined to be the same as $L^\infty(\R^n)$ and,
for any $f\in L^\infty_w(\rn)$, let
$$\|f\|_{L^\infty_w(\rn)}:= \|f\|_{L^\infty(\rn)}.$$

Let $\phi$ be a function in the Schwartz class, $\mathcal S(\bR^n)$,
satisfying $\phi(x)=1$ for all $x\in B(0,1)$.
The {\it maximal function} of a tempered distribution $f\in \mathcal S'(\bR^n)$
is defined by
\begin{equation}
\mv f:= \sup_{t\in(0,\fz)}|f*\phi_t|,
\end{equation}
where  $\phi_t(\cdot):=\frac1{t^n}\phi(t^{-1}\cdot)$ for all $t\in (0,\infty)$.
Then the {\it weighted Hardy space} $H^1_w(\bR^n)$ is defined as the space of all tempered distributions $f\in \mathcal S'(\bR^n)$ such that
$$\|f\|_{H^1_w(\rn)}:=\|\mv f\|_{L^1_w(\rn)}<\fz;$$
see \cite{g79}.

Notice that $\|\cdot\|_{H^1_w(\rn)}$ defines a norm on $H^1_w(\bR^n)$,
whose size depends on the choice of $\phi$,
but the space $H^1_w(\bR^n)$ is independent of this choice.

\begin{definition}\label{d-bmow}
Let $w\in A_\infty(\rn)$ and
$\int_{\bR^n} \frac{w(x)}{1+|x|^n}\,dx<\infty$.
A locally integrable function $b$ is said to be in
$\mathcal{BMO}_w(\rn)$ if
\begin{equation}
\|b\|_{\mathcal{BMO}_w(\rn)}:= \sup_{B}
\left\{\int_{B^\com} \frac{w(x)}{|x-x_B|^n}\,dx \frac{1}{w(B)}
\int_{B} |b(x) - b_B|dx \right\}<\infty,
\end{equation}
where the supremum is taken over all balls $B\subset \R^n$ and $B^\com:=\rn\bh B$.
Here and hereafter, $x_B$ denotes the center of ball $B$,
$$w(B):=\int_B w(x)\,dx\quad \mathrm{and}\quad b_B:=\frac1{|B|}\int_Bb(x)\,dx.$$
\end{definition}

It should be pointed out that the space $\mathcal{BMO}_w(\rn)$ has been
considered first by Bloom \cite{Bl} when studying the
pointwise multipliers of weighted BMO spaces (see also \cite{Ya}).

Recall that a locally integrable function $b$ is said to be in $\BMO$ if
$$\|b\|_{\BMO}:=\sup_{B}\frac1{|B|}\int_B|b(x)-b_B|\,dx<\fz,$$
where the supremum is taken over all balls $B\subset \R^n$.

\begin{remark}
	{\rm (i)}   $\mathcal{BMO}_w(\rn)\subset{\rm BMO}(\bR^n)$ and the inclusion is continuous (see Proposition \ref{inclusions between BMO spaces} of Section \ref{s2}).

{\rm (ii)} It is easy to show that, when $n=1$, $w(x):=|x|^{-1/2}\in A_1(\rr)$
and $\int_\rr\frac{w(x)}{1+|x|}\,dx<\fz$. Let
\begin{eqnarray*}
f(x):= \left\lbrace
\begin{array}{l l}
|1-x|,
\ &|x|\le1,\\
\\
0,
&|x|>1.
\end{array} \right.
\end{eqnarray*}
Then $f\in\mathcal{BMO}_w(\rn)$,
which implies that $\bmow$ is not a trivial function space.
\end{remark}

To state our main results, we first recall the definition
of Calder\'on-Zygmund operators.
For $\dz\in(0,1]$, a linear operator $T$ is called a
\emph{$\dz$-Calder\'on-Zygmund
operator} if $T$ is a linear bounded operator on $L^2(\rn)$
and there exist a kernel $K$ on
$(\rn\times\rn)\setminus\{(x,x):\ x\in\rn\}$ and a positive constant $C$ such that,
for all $x,\,y,\,z\in\rn$,
$$|K(x,y)|\le \frac C{|x-y|^{n}}\quad\mathrm{if}\quad x\neq y,$$
$$|K(x,y)-K(x,z)| + |K(y,x) - K(z,x)|\le C\frac{|y-z|^\dz}{|x-y|^{n+\dz}}
\quad\mbox{ if }\quad |x-y|>2|y-z|$$
and, for all $f\in L^2(\rn)$ with compact support and $x\notin\supp(f)$,
$$Tf(x)=\int_{\supp (f)}K(x,y)f(y)\,dy.$$

The main result of this paper is the following theorem.

\begin{theorem}\label{t-bmow}
Let $\dz\in(0,1]$, $w\in A_{1+\dz/n}(\rn)$
with $\int_{\bR^n} \frac{w(x)}{1+|x|^n}\,dx<\infty$ and $b\in\BMO$.
Then the following two statements are equivalent:
\begin{enumerate}
\item[{\rm (i)}]
for every $\delta$-Calder\'on-Zygmund operator $T$, the commutator $[b,T]$ is bounded from $H^1_w(\bR^n)$ into $L^1_w(\bR^n)$;

\item[{\rm (ii)}] $b\in \mathcal{BMO}_w(\rn)$.
\end{enumerate}
\end{theorem}

\begin{remark}
When $w(x)\equiv1$ for all $x\in\rn$,
we see that $\int_\rn\frac{1}{1+|x|^n}\,dx=\fz$
and hence, in this case, $\bmow$ can be seen as a zero space in $\BMO$.
In this case, Theorem \ref{t-bmow} coincides with the result in \cite{HST}.
\end{remark}

The next theorem gives a sufficient condition
of the boundedness of $[b,T]$ on $H_w^1(\rn)$.
Recall that, for $w\in A_p(\R^n)$ with $p\in(1,\fz)$ and
$q\in [p, \infty]$,
a measurable function $a$ is called an {\it $(H_w^1(\rn),q)$-atom} related to a ball
$B\st\rn$ if
 \begin{enumerate}
 	\item[(i)] $\supp a\st B$,
 	
 	\item[(ii)] $\int_\rn a(x)\,dx=0$,
 	
 	\item[(iii)] $\|a\|_{L^q_w(\rn)}\le [w(B)]^{1/q-1}$
 \end{enumerate}
and also that $T^*1=0$ means $\int_{\mathbb R^n} Ta(x)\,dx=0$
holds true for all $(H_w^1(\rn),q)$-atoms $a$.

\begin{theorem}\label{t-hbd}
Let $\dz\in(0,1]$, $T$ be a $\delta$-Calder\'on-Zygmund operator,
$w\in A_{1+\dz/n}(\rn)$
with $\int_{\bR^n} \frac{w(x)}{1+|x|^n}\,dx<\infty$ and $b\in\mathcal{BMO}_w(\rn)$.
If $T^*1=0$, then the commutator $[b,T]$ is bounded on $H^1_w(\bR^n)$, namely, 
there exists a positive constant $C$ such that, for all $f\in H^1_w(\bR^n)$,
$$\|[b,T](f)\|_{H^1_w(\bR^n)}\le C\|f\|_{H^1_w(\bR^n)}.$$
\end{theorem}

Finally we make some conventions on notation.
Throughout the whole article, we denote by $C$ a \emph{positive constant} which is
independent of the main parameters, but it may vary from line to
line. The {\it symbol} $A\ls B$ means that $A\le CB$. If $A\ls
B$ and $B\ls A$, then we write $A\sim B$.
For any measurable subset $E$ of $\rn$, we denote by $E^\complement$ the
{\it set} $\rn\setminus E$ and its \emph{characteristic function} by $\chi_{E}$.
We also let $\nn:=\{1,\,2,\,
\ldots\}$ and $\zz_+:=\nn\cup\{0\}$.


\section{Proofs of Theorems \ref{t-bmow} and \ref{t-hbd}}\label{s2}

\hskip\parindent
We begin with pointing out that, if $w\in A_\fz(\rn)$,
then there exist $p,\,r\in(1,\fz)$ such that
$w\in A_p(\rn)\cap RH_r(\rn)$,
where $RH_r(\rn)$ denotes the {\it reverse H\"older class} of weights $w$
satisfying that there exists a positive constant $C$ such that
$$\left(\frac{1}{|B|}\int_B [w(x)]^r\,dx\right)^{1/r}\le C \frac{1}{|B|}\int_B w(x)\,dx$$
for every ball $B\subset \R^n$. Moreover, there exist positive constants $C_1\le C_2$,
depending on $[w]_{A_\fz(\rn)}$, such that,
for any measurable sets $E\subset B$,
\begin{equation}\label{fundamental estimates for Muckenhoupt weights}
C_1\left(\frac{|E|}{|B|}\right)^p \le \frac{w(E)}{w(B)}\le C_2\left(\frac{|E|}{|B|}\right)^{(r-1)/r}.
\end{equation}

In order to prove Theorems \ref{t-bmow} and \ref{t-hbd},
we need the following proposition and several technical lemmas.

\begin{proposition}\label{inclusions between BMO spaces}
Let $w\in A_\infty(\R^n)$. Then there exists a positive constant $C$ such that,
for any $f\in \mathcal{BMO}_w(\rn)$,
$$\|f\|_{\BMO}\le C \|f\|_{\mathcal{BMO}_w(\rn)}.$$
\end{proposition}

\begin{proof}
By (\ref{fundamental estimates for Muckenhoupt weights}), for any ball $B\st\rn$,
we have
\begin{eqnarray*}
	\int_{B^\com} \frac{w(x)}{|x-x_B|^n}\,dx \frac{1}{w(B)}
	&&\ge\int_{2B\bh B} \frac{w(x)}{|x-x_B|^n}\,dx \frac{1}{w(B)}\\
	&&\ge\frac{w(2B\bh B)}{|2B|}\frac1{w(B)}\\
    &&\gs\frac1{|B|}.
\end{eqnarray*}
This proves that $\|f\|_{\BMO}\ls \|f\|_{\mathcal{BMO}_w(\rn)},$
which completes the proof of Proposition \ref{inclusions between BMO spaces}.
\end{proof}

\begin{lemma}\label{l-con}
	Let $f$ be a measurable function such that supp $f\subset B:=B(x_0,r)$
	with $x_0\in\rn$ and $r\in(0,\fz)$. Then there exists a positive constant $C:=C(\phi,n)$, depending only on $\phi$ and $n$, such that, for all $x\notin B$,
	$$\frac{1}{|x-x_0|^n} \left|\int_{B(x_0,r)} f(y)\,dy \right|\le C \mv f(x).$$
\end{lemma}

\begin{proof}
	For $x\notin B(x_0,r)$ and any $y\in B(x_0,r)$,
	it follows that 
$$\frac{|x-y|}{2|x-x_0|} < \frac{|x-x_0|+r}{2|x-x_0|}\le 1,$$
	which, together with $\phi\equiv1$ on $B(0,1)$,
	further implies that $\phi(\frac{x-y}{2|x-x_0|})=1$.
	Thus, we know that
	\begin{eqnarray*}
		\mv f(x)
		=&&\sup_{t\in(0,\fz)}|f*\phi_t(x)|\ge|f*\phi_{2|x-x_0|}(x)|\\
		=&&\frac{1}{2^n|x-x_0|^n}\left|\int_{B(x_0,r)} f(y)
		\phi\left(\frac{x-y}{2|x-x_0|}\right)\,dy\right|\\
		\gs&&\frac{1}{|x-x_0|^n} \left|\int_{B(x_0,r)} f(y)\,dy\right|,
	\end{eqnarray*}
	which completes the proof of Lemma \ref{l-con}.
\end{proof}

\begin{lemma}\label{a John-Nirenberg type lemma for weighted BMO}
	Let $w\in A_\infty(\R^n)$ and $q\in [1,\infty)$. Then there exists a positive constant $C$ such that, for any $f\in{\rm BMO}(\R^n)$ and any ball $B\subset \R^n$,
	$$\left[\frac{1}{w(B)} \int_{B} |f(x) - f_B|^q w(x)\,dx\right]^{1/q}\leq C \|f\|_{\BMO}.$$
\end{lemma}

\begin{proof}
It follows from the John-Nirenberg inequality that there exist two positive
constants $c_1$ and $c_2$, depending only on $n$, such that, for all $\lambda>0$,
$$|\{x\in B: |f(x)- f_B|>\lambda\}|
\leq c_1 e^{-c_2 \frac{\lambda}{\|f\|_{\BMO}}} |B|;$$
see \cite{jn}.
Therefore, by (\ref{fundamental estimates for Muckenhoupt weights}), we see that
\begin{eqnarray*}
\frac{1}{w(B)} \int_{B} |f(x) - f_B|^q w(x)\,dx &=& q \int_{0}^\infty \lambda^{q-1} \frac{w(\{x\in B: |f(x)- f_B|>\lambda\})}{w(B)} \,d\lambda\\
&\ls&   \int_0^\infty \lambda^{q-1}
\left[\frac{|\{x\in B: |f(x)- f_B|>\lambda\}|}{|B|}\right]^{(r-1)/r} \,d\lambda\\
&\ls&  \int_0^\infty \lambda^{q-1} e^{-c_2 \frac{r-1}{r} \frac{\lambda}{\|f\|_{\BMO}}}
\,d\lambda\\
&\ls& \|f\|^q_{\BMO},
\end{eqnarray*}
which completes the proof of Lemma \ref{a John-Nirenberg type lemma for weighted BMO}.	
\end{proof}

\begin{lemma}\label{the class of K type}
Let $\dz\in(0,1]$, $q\in (1, 1+\delta/n)$ and $w\in A_{q}(\rn)$.  Assume that $T$ is a $\dz$-Calder\'on-Zygmund operator.
Then there exists a positive constant $C$ such that,
for any $b\in{\rm BMO}(\R^n)$ and $(H_w^1(\rn),q)$-atom $a$ related to the ball $B\subset \R^n$,
$$\|(b-b_B)Ta\|_{L^1_w(\rn)}\leq C \|b\|_{\BMO}.$$
\end{lemma}

\begin{proof}
It suffices to show that
$$\mi_1:= \int_{2B} |[b(x)-b_B]Ta(x)| w(x)\,dx \ls \|b\|_{\BMO}$$
and
$$\mi_2:= \int_{(2B)^{\com}} |[b(x)-b_B]Ta(x)| w(x)\,dx \ls \|b\|_{\BMO}.$$

Indeed, by the boundedness of $T$ from $H^1_w(\rn)$ to $L^1_w(\rn)$ and from
$L^q_w(\rn)$ to itself with $q\in(1,1+\dz/n)$ (see \cite[Theorem 2.8]{gk94}),
the H\"older inequality and Lemma \ref{a John-Nirenberg type lemma for weighted BMO},
we conclude that
\begin{eqnarray}\label{bt2}
\mi_1 &=& \int_{2B} |[b(x)-b_B]Ta(x)| w(x)\,dx \\
&\leq& |b_{2B} - b_B| \|Ta\|_{L^1_w(\rn)} + \int_{2B} |[b(x)-b_{2B}]Ta(x)| w(x)\,dx \noz\\
&\ls& \|b\|_{\BMO} + \lf[\int_{2B} |b(x)-b_{2B}|^{q'} w(x)\,dx\r]^{1/q'} \lf[\int_{2B} |Ta(x)|^q w(x)\,dx\r]^{1/q}\noz\\
&\ls& \|b\|_{\BMO} + [w(2B)]^{1/q'} \|b\|_{\BMO} \|a\|_{L^q_w(\rn)}\noz\\
&\ls&  \|b\|_{\BMO},\noz
\end{eqnarray}
here and hereafter, $1/q'+ 1/q=1$.

On the other hand, by the H\"older inequality, (\ref{w-2}), Lemma \ref{a John-Nirenberg type lemma for weighted BMO} and (\ref{fundamental estimates for Muckenhoupt weights}),
we know that
\begin{eqnarray}\label{bt3}
\mi_2&&=\int_{(2B)^\com} |[b(x)-b_B]Ta(x)| w(x)\,dx\\
&&=\int_{(2B)^\com} |b(x)-b_B|
\lf|\int_Ba(y)[K(x,y)-K(x,x_0)]\,dy\r| w(x)\,dx\noz\\
&&\le\int_B|a(y)|\int_{(2B)^\com} |b(x)-b_B|
\lf|K(x,y)-K(x,x_0)\r| w(x)\,dx\,dy\noz\\
&&=\int_B|a(y)|
\sum_{k=1}^\fz\int_{2^{k+1}B\setminus 2^k B} |b(x)-b_B|
\lf|K(x,y)-K(x,x_0)\r| w(x)\,dx\,dy\noz\\
&&\ls\int_B|a(y)|\,dy
\sum_{k=1}^\fz\int_{2^{k+1}B\setminus 2^k B}
\frac{r^\dz}{(2^kr)^{n+\dz}}|b(x)-b_B| w(x)\,dx\noz\\
&&\ls \left[\int_B |a(y)|^q w(y)\,dy\right]^{1/q} \left[\int_B [w(y)]^{-q'/q}\,dy\right]^{1/q'}  \noz\\
&&\hskip 1cm \times\sum_{k=1}^\fz2^{-k\dz}\frac{1}{|2^{k+1}B|}\int_{2^{k+1}B}
\left[|b(x)-b_{2^{k+1}B}|+|b_{2^{k+1} B}-b_B|\right] w(x)\,dx\noz\\
&&\ls\frac{|B|}{w(B)}
\sum_{k=1}^\fz2^{-k\dz}k\frac{w(2^{k+1}B)}{|2^{k+1}B|}\|b\|_{\BMO}\noz\\
&&\ls
\|b\|_{\BMO} \sum_{k=1}^\fz k 2^{-k[\delta+n -nq]}\noz\\
&&\ls \|b\|_{\BMO},\noz
\end{eqnarray}
since $\delta+n -nq >0$ and $|b_{2^{k+1} B}-b_B|\ls k \|b\|_{\BMO}$ for all $k\geq 1$.

Combining (\ref{bt2}) and (\ref{bt3}),
we then complete the proof of Lemma \ref{the class of K type}.
 \end{proof}

 The following lemma is due to Bownik et al. \cite[Theorem 7.2]{BLYZ}.

 \begin{lemma}\label{blyz}
 Let $w\in A_{1+\delta/n}(\R^n)$ and $\cx$ be a Banach space.
 Assume that $T$ is a linear operator defined on the space of finite linear combinations of continuous $(H_w^1(\rn),\infty)$-atoms with the property that
 $$\sup\left\{\|T(a)\|_{\cx}: \mbox{$a$ is a continuous $(H_w^1(\rn),\infty)$-atom} \right\}<\infty.$$
 Then $T$ admits a unique continuous extension to a bounded linear operator from $H^1_w(\R^n)$ into $\cx$.
 \end{lemma}

 Let $w\in A_{1+\delta/n}(\R^n)$ and $\varepsilon\in(0,\fz)$. Recall that $m$ is called an {\it $(H_w^1(\rn),\infty,\varepsilon)$-molecule} related to the ball $B\subset \R^n$ if
 \begin{enumerate}
  \item[(i)] $\int_{\bR^n} m(x)dx=0$,

 \item[(ii)] $\|m\|_{L^\infty(S_j)}\leq 2^{-j\varepsilon} [w(S_j)]^{-1}$,
 $j\in\zz_+$,
 where $S_0=B$ and $S_j= 2^{j+1}B\setminus 2^jB$ for $j\in\nn$.
\end{enumerate}

 \begin{lemma}\label{molecule}
 	Let $w\in A_{1+\delta/n}(\R^n)$ and $\varepsilon>0$. Then there exists a positive constant $C$ such that, for any $(H_w^1(\rn),\infty,\varepsilon)$-molecule $m$ related to the ball $B$, it holds true that
 	$$m= \sum_{j=0}^\infty \lambda_j a_j,$$
 	where $\{a_j\}_{j=0}^\fz$ are $(H_w^1(\rn),\infty)$-atoms  related to the balls
	$\{2^{j+1}B\}_{j\in\zz_+}$
 	and there exists a positive constant $C$ such that
	$|\lambda_j|\le C 2^{-j\varepsilon}$ for all $j\in\zz_+$.
 \end{lemma}

 \begin{proof}
 	The proof of this lemma is standard (see, for example, \cite[Theorem 4.7]{SY}),
	the details being omitted.
 \end{proof}

Now we are ready to give the proofs of Theorems \ref{t-bmow} and \ref{t-hbd}.

\begin{proof}[\bf Proof of Theorem \ref{t-bmow}]
First, we prove that (ii) implies (i). Since $w\in A_{1+\delta/n}(\R^n)$,
it follows that there exists $q\in (1,1+\delta/n)$ such that $w\in A_{q}(\R^n)$. By Lemma \ref{blyz}, it suffices to prove that,
for any continuous $(H_w^1(\rn),\infty)$-atom $a$ related to the ball $B=B(x_0,r)$
with $x_0\in\rn$ and $r\in(0,\fz)$,
\begin{equation}\label{bt0}
\|[b,T](a)\|_{L^1_w(\rn)}\ls \|b\|_{\mathcal{BMO}_w(\rn)}.
\end{equation}

By Lemma \ref{the class of K type} and the boundedness of $T$ from $H^1_w(\rn)$ to $L^1_w(\rn)$, (\ref{bt0}) is reduced to showing that
\begin{equation}\label{the key estimate}
\|(b-b_B)a\|_{H^1_w(\rn)}\ls \|b\|_{\mathcal{BMO}_w(\rn)}.
\end{equation}
To do this, for every $x\in (2B)^\com$ and $y\in B$,
we see that $|x-y|\sim|x-x_0|$ and
\begin{eqnarray*}
\mv ([b-b_B]a)(x)
&&\ls\sup_{t\in(0,\fz)}\frac{1}{t^n}\int_B\int_{B}|b(y) - b_B||a(y)|
\left|\phi\left(\frac{x-y}t\right)\right|dy\\
&&\ls\frac{1}{|x-x_0|^n}\int_{B}|b(y) - b_B||a(y)|\,dy.
\end{eqnarray*}
Hence
$$\int_{(2B)^\com} \mv ([b-b_B]a)(x) w(x)\,dx\ls \|b\|_{\mathcal{BMO}_w(\rn)}.$$
In addition, by the boundedness of $\mv$ on $L^q_w(\rn)$ with $q\in(1,1+\delta/n)$,
Lemma \ref{a John-Nirenberg type lemma for weighted BMO} and Proposition \ref{inclusions between BMO spaces}, we know that
\begin{eqnarray*}
\int_{2B} \mv ([b-b_B]a)(x) w(x)\,dx &&\ls w(2B)^{1/q'} \|(b-b_B)a\|_{L^q_w(\rn)}\\
&&\ls \left[ \frac{1}{w(B)}\int_B |b(x) -b_B|^q w(x)\,dx \right]^{1/q}\\
&&\ls \|b\|_{\BMO}\\
&&\ls\|b\|_{\mathcal{BMO}_w(\rn)},
\end{eqnarray*}
which concludes the proof of (ii) implying (i).

We now prove that (i) implies (ii).
Let $\{R_j\}_{j=1}^n$ be the classical Riesz transforms.
Then, by Lemma \ref{the class of K type}, we find that,
for any $(H_w^1(\rn),\infty)$-atom $a$ related to the ball $B$ and $j\in\{1,\ldots,n\}$,
\begin{eqnarray*}
\|R_j([b-b_B]a)\|_{L^1_w(\rn)} &&\leq \|[b,R_j](a)\|_{L^1_w(\rn)} + \|(b-b_B)R_ja\|_{L^1_w(\rn)}\\
&&\ls \|[b,R_j]\|_{H^1_w(\rn)\to L^1_w(\rn)} + \|b\|_{\BMO},
\end{eqnarray*}
here and hereafter,
$$\|[b,R_j]\|_{H^1_w(\rn)\to L^1_w(\rn)}
:=\sup_{\|f\|_{H^1_w(\rn)}\le 1}\|[b,R_j]f\|_{L^1_w(\rn)}.$$
By the Riesz transform characterization of $H_w^1(\rn)$ (see \cite{Wh76}),
we see that $(b-b_B)a\in H^1_w(\bR^n)$ and, moreover,
\begin{equation}\label{2.7}
\|(b-b_B)a\|_{H^1_w(\rn)}\ls \|b\|_{\BMO} + \sum_{j=1}^n \|[b,R_j]\|_{H^1_w(\rn)\to L^1_w(\rn)}.
\end{equation}

For any ball $B:= B(x_0,r)\subset\bR^n$ with $x_0\in\rn$ and $r\in(0,\fz)$,
let 
$$a:= \frac{1}{2w(B)}(f-f_B)\chi_B,$$ 
where $f:=\sign (b-b_B)$.
It is easy to see that $a$ is an $(H_w^1(\rn),\infty)$-atom related to the ball $B$. Moreover, for every $x\notin B$, Lemma \ref{l-con} gives us that
\begin{eqnarray*}
\frac{1}{|x-x_0|^n} \frac{1}{2w(B)} \int_B |b(x)- b_B|\,dx
&&=\frac{1}{|x-x_0|^n} \int_{B}(b(x)-b_B)a(x)\,dx\\
&&\ls\mv ([b-b_B]a)(x).
\end{eqnarray*}
This, together with (\ref{2.7}), allows to conclude that $b\in \mathcal{BMO}_w(\rn)$ and, moreover,
$$\|b\|_{\mathcal{BMO}_w(\rn)}\ls \|b\|_{\BMO} + \sum_{j=1}^n \|[b,R_j]\|_{H^1_w(\rn)\to L^1_w(\rn)},$$
which complete the proof of Theorem \ref{t-bmow}.
\end{proof}

\begin{proof}[\bf Proof of Theorem \ref{t-hbd}]
	By Lemma \ref{blyz}, it suffices to prove that,
	for any continuous $(H_w^1(\rn),\infty)$-atom $a$ related to the ball $B$,
	\begin{equation}\label{Hardy type estimate for atoms}
	\|[b,T](a)\|_{H^1_w(\rn)}\ls \|b\|_{\mathcal{BMO}_w(\rn)}.
	\end{equation}
	
 By (\ref{the key estimate}) and the boundedness of $T$ on $H^1_w(\rn)$
 (see \cite[Theorem 1.2]{Ky1}), (\ref{Hardy type estimate for atoms}) is reduced to
 proving that
$$\|(b-b_B)Ta\|_{H^1_w(\rn)}\ls \|b\|_{\mathcal{BMO}_w(\rn)}.$$

Since $w\in A_{1+\delta/n}(\R^n)$, it follows that
there exists $q\in (1,1+\delta/n)$ such that $w\in A_q(\R^n)$.
By this and the fact that $T$ is a $\delta$-Calder\'on-Zygmund operator,
together with a standard argument, we find that $Ta$ is an
$(H_w^1(\rn),\infty,\varepsilon)$-molecule related to the ball $B$ with
$\varepsilon:= n+\delta-nq>0$. Therefore, by Lemma \ref{molecule}, we have
$$Ta= \sum_{j=0}^\infty \lambda_j a_j,$$
where $\{a_j\}_{j=0}^\fz$ are $(H_w^1(\rn),\infty)$-atoms
related to the balls $\{2^{j+1}B\}_{j=0}^\fz$ and $|\lambda_j|\ls 2^{-j \varepsilon}$
for all $j\in\zz_+$.
Thus, by (\ref{the key estimate}) and Proposition \ref{inclusions between BMO spaces}, we obtain
\begin{eqnarray*}
\|(b-b_B)Ta\|_{H^1_w(\rn)}
 &&\leq \sum_{j=0}^\infty |\lambda_j|
 \lf[\|(b-b_{2^{j+1}B})a_j\|_{H^1_w(\rn)} +\|(b_{2^{j+1}B}-b_B)a_j\|_{H^1_w(\rn)} \r]\\
&&\ls \|b\|_{\mathcal{BMO}_w(\rn)}\sum_{j=0}^\infty 2^{-j\varepsilon}
+ \|b\|_{\BMO}\sum_{j=0}^\infty (j+1) 2^{-j\varepsilon}\\
&&\ls \|b\|_{\mathcal{BMO}_w(\rn)},
\end{eqnarray*}
which completes the proof of (i) implying (ii)
and hence Theorem \ref{t-hbd}.
\end{proof}

{\bf Acknowledgements.} The paper was completed when the second author was visiting
to Vietnam Institute for Advanced Study in Mathematics (VIASM), who would like
to thank the VIASM for its financial support and hospitality.

\bigskip

\noindent Yiyu Liang

\medskip

\noindent Department of Mathematics,
Beijing Jiaotong University,
Beijing 100044, People's Republic of China

\smallskip

\noindent {\it E-mail:} \texttt{yyliang@bjtu.edu.cn}

\bigskip

\noindent Luong Dang Ky

\medskip

\noindent Department of Mathematics,
University of Quy Nhon,
170 An Duong Vuong,
Quy Nhon, Binh Dinh, Vietnam

\smallskip

\noindent {\it E-mail}: \texttt{dangky@math.cnrs.fr}

\bigskip

\noindent Dachun Yang {\it(Corresponding Author)}

\medskip

\noindent School of Mathematical Sciences, Beijing Normal University,
Laboratory of Mathematics and Complex Systems, Ministry of
Education, Beijing 100875, People's Republic of China

\smallskip

\noindent {\it E-mail:} \texttt{dcyang@bnu.edu.cn}

\end{document}